\documentclass[11pt,twoside, leqno]{article}

\usepackage{mathrsfs}
\usepackage{amssymb}
\usepackage{amsmath}
\usepackage{amsthm}
\usepackage{amsfonts}
\usepackage{color}

\allowdisplaybreaks

\def\rr{{\mathbb R}}
\def\rn{{{\rr}^n}}

\def\fz{\infty}

\def\az{\alpha}

\renewcommand\tilde{\widetilde}

\def\supp{{\rm{\,supp\,}}}

\def\ls{\lesssim}
\def\lz{\lambda}

\def\bz{\beta}

\def\gz{{\gamma}}

\def\sz{\sigma}

\def\hs{\hspace{0.3cm}}

\def\r{\right}
\def\lf{\left}

\def\bint{{\ifinner\rlap{\bf\kern.30em--}
\int\else\rlap{\bf\kern.35em--}\int\fi}\ignorespaces}

\def\sbint{{\ifinner\rlap{\bf\kern.32em--}
\hspace{0.078cm}\int\else\rlap{\bf\kern.45em--}\int\fi}\ignorespaces}

\newtheorem{theorem}{Theorem}[section]
\newtheorem{lemma}[theorem]{Lemma}
\newtheorem{corollary}[theorem]{Corollary}

\theoremstyle{definition}

\newtheorem{definition}[theorem]{Definition}
\numberwithin{equation}{section}

\numberwithin{equation}{section}

\numberwithin{equation}{section}

\begin{document}

\arraycolsep=1pt

\title{\Large\bf Boundedness of parameterized Marcinkiewicz integrals with variable kernels on Hardy spaces and weak Hardy spaces
\footnotetext{\hspace{-0.35cm}
{\it 2010 Mathematics Subject Classification}.
{Primary 42B25; Secondary 42B30.}
\endgraf{\it Key words and phrases.}  Marcinkiewicz integral, Hardy space, variable kernel.
}}
\author{Li Bo
}
\date{  }
\maketitle

\vspace{-0.8cm}

\begin{minipage}{14.5cm}\small
{
\noindent
{\bf Abstract.}
In this paper, by using the atomic decomposition theory of Hardy space and weak Hardy space,
the author establishes the boundedness of parameterized Marcinkiewicz integral with variable kernel
on these spaces, under the Dini condition or H\"{o}rmander condition imposed on kernel.
}
\end{minipage}



\section{Introduction\label{s1}}
Let $S^{n-1}$ be the unit sphere in the $n$-dimensional Euclidean space $\rn \ (n\ge2)$ with normalized Lebesgue measure $d\sigma$.
A function $\Omega(x,\,z)$ defined on $\rn\times\rn$ is said to be in $L^\fz(\rn)\times L^q(S^{n-1})$ with $q\ge1$,
if $\Omega(x,\,z)$ satisfies the following conditions:
\begin{align}\label{e1.1}
\Omega(x,\,\lz z)=\Omega(x,\,z) \ \mathrm{for} \ \mathrm{any} \ x,\,z\in\rn \ \mathrm{and} \ \lz\in(0,\,\,\fz),
\end{align}
\begin{align}\label{e1.2}
\int_{S^{n-1}}\Omega(x,\,z)\,d\sigma(z')=0 \ \mathrm{for} \ \mathrm{any} \ x\in\rn,
\end{align}
\begin{align}\label{e1.3}
\sup_{\substack{x\in\rn \\ r\ge0}}\lf(\int_{S^{n-1}}|\Omega(x+rz',\,z')|^q\,d\sigma(z')\r)^{1/q}<\fz,
\end{align}
where $z':=z/{|z|}$ for any $z\neq{\mathbf{0}}$.
The {\it {singular integral operator with variable kernel}} is defined by
\begin{align*}
T_\Omega(f)(x):=\mathrm{p.v.}\int_\rn \frac{\Omega(x,\,x-y)}{|x-y|^n}f(y)\,dy.
\end{align*}

In 1955 and 1956, Calder\'{o}n and Zygmund \cite{cz55,cz56} investigated the $L^p$ boundedness of $T_\Omega$.
They found that these operators are closely related to the problem about the second-order linear elliptic
equations with variable coefficients.
In 2011, Chen and Ding \cite{cd11} consider the same problem for
the {\it{parameterized Marcinkiewicz integral with variable kernel}} $\mu_\Omega^\rho$ defined by
\begin{align*}
\mu_\Omega^\rho(f)(x):=\lf(\int_0^\fz\lf|\int_{|x-y|\le t}
 \frac{\Omega(x,\,x-y)}{|x-y|^{n-\rho}}f(y)\,dy\r|^2\frac{dt}{t^{2\rho+1}}\r)^{1/2},
\end{align*}
where $0<\rho<n$.
For any $1\le q\le\fz$, $q'$ denotes the {\it{conjugate index}} of $q$, namely, $1/q+1/{q'}=1$.
They obtained the following result:
\begin{flushleft}
{\bf{Theorem A.}}
{\it{
Let $0<\rho<n$, $1<p\le2$ and $q>p'(n-1)/n$.
If $\Omega(x,z)\in L^\fz(\rn)\times L^q(S^{n-1})$, then $\mu_\Omega^\rho$ is bounded on $L^p(\rn).$
}}
\end{flushleft}

On the other hand, as everyone knows, many important operators
are better behaved on Hardy space $H^p(\rn)$ than on Lebesgue $L^p(\rn)$
space in the range $p\in(0,\,1]$.
For example, when $p\in(0,\,1]$, the Riesz transforms are bounded on the Hardy space $H^p(\rn)$,
but not on the corresponding Lebesgue space $L^p(\rn)$.
Therefore, one can consider $H^p(\rn)$ to be a very natural replacement for $L^p(\rn)$ when $p\in(0,\,1]$.
We refer to \cite{l95} for a complete survey of the real-variable theory of Hardy space.
Motivated by this, it is a natural and interesting problem to ask,
when $p\in(0,\,1]$, whether the operator $\mu_\Omega^\rho$ is bounded
from Hardy space $H^p(\rn)$ to Lebesgue space $L^p(\rn)$.
In this paper we shall answer this problem affirmatively. Not only that,
we also discuss boundedness of $\mu_\Omega^\rho$ from weak Hardy space $WH^p(\rn)$
to weak Lebesgue space $WL^p(\rn)$.

Precisely, the present paper is built up as follows.
In next section, we first recall a notion concerning variable kernel $\Omega$.
Then we discuss the boundedness of $\mu_\Omega^\rho$ from Hardy space to Lebesgue space
(see Theorems \ref{dl.1}-\ref{dl.3} below).
Section \ref{s3} is devoted to establishing the boundedness of $\mu_\Omega^\rho$ from weak Hardy space to weak Lebesgue space
(see Theorems \ref{dl.4}-\ref{dl.6} below).
As corollary, we obtain that $\mu_\Omega^\rho$ is also of the weak type (1,\,1) (see Corollary \ref{tl.7} below).
%
Throughout this paper the letter $C$ will denote a \emph{positive constant} that may vary
from line to line but will remain independent of the main variables.
The \emph{symbol} $P\ls Q$ stands for the inequality $P\le CQ$. If $P\ls Q\ls P$, we then write $P\sim Q$.




\section{$H^p$-$L^p$ boundedness}\label{s2}

Before stating the main results of this scetion, we recall a notion about
the variable kernel $\Omega(x,\,z)$.
For any $0<\alpha\le1$, a function $\Omega(x,z)\in L^\fz(\rn)\times L^1(S^{n-1})$ is said to satisfy the
{\it{$L^{1,\,\az}$-Dini condition}} (when $\alpha=0$, it is called the {\it{$L^1$-Dini condition}}) if
\begin{align*}
\int_0^1\frac{\omega(\delta)}{\delta^{1+\alpha}}\,d\delta<\fz,
\end{align*}
where ${\omega(\delta)}$ is the integral modulus of continuity of $\Omega$ defined by
\begin{align*}
{\omega(\delta)}:=\sup_{\substack{x\in\rn \\ r\ge0}}
\lf(\int_{S^{n-1}}\sup_{\substack{y'\in S^{n-1} \\ |y'-z'|\le\delta}}
\lf|\Omega(x+rz',\,y')-\Omega(x+rz',\,z')\r|\,d\sigma(z')\r).
\end{align*}

The main results of this section are as follows.
\begin{theorem}\label{dl.1}
Let $0<\rho<n$, $q>2(n-1)/n$, $0<\az\le1$, $\bz:=\min\{\az,\,1/2\}$ and $n/(n+\bz)<p<1$.
Suppose that $\Omega\in L^\fz(\rn) \times L^q(S^{n-1})$ satisfies the $L^{1,\,\az}$-Dini condition.
Then $\mu_\Omega^\rho$ is bounded from $H^p(\rn)$ to $L^p(\rn)$.
\end{theorem}

\begin{theorem}\label{dl.2}
Let $0<\rho<n$ and $q>2(n-1)/n$. Suppose that $\Omega\in L^\fz(\rn) \times L^q(S^{n-1})$.
If there exists a positive constant ${C}$ such that,
for any $y\in\rn$,
\begin{align}\label{p1}
\int_{|x|\geq2|y|}\lf|\frac{\Omega(x,\,x-y)}{|x-y|^n}-\frac{\Omega(x,\,x)}{|x|^n}\r|\,dx\leq {C},
\end{align}
then $\mu_\Omega^\rho$ is bounded from $H^1(\rn)$ to $L^1(\rn)$.
\end{theorem}

\begin{theorem}\label{dl.3}
Let $0<\rho<n$ and $q>2(n-1)/n$. Suppose that $\Omega\in L^\fz(\rn) \times L^q(S^{n-1})$ satisfies the $L^1$-Dini condition.
Then $\mu_\Omega^\rho$ is bounded from $H^1(\rn)$ to $L^1(\rn)$.
\end{theorem}

To show the above theorems, we need the following definition and lemma.
\begin{definition}\label{d2.11}{\rm{(\cite{l95})}}
Let $0<p\le1$ and the nonnegative integer $s\ge \lfloor n(1/p-1)\rfloor$.
A function $a(x)$ is called a {\it $(p,\,\fz,\,s)$-atom} associated with some ball $B\subset\rn$
if it satisfies the following three conditions:
\begin{enumerate}
\item[\rm{(i)}] $a$ is supported in $B$;
\item[\rm{(ii)}] $\|a\|_{L^\fz(\rn)}\leq |B|^{-1/p}$;
\item[\rm{(iii)}] $\int_\rn a(x)x^\gz dx=0$ for any multi-index $\gz$ with $|\gz|\leq s$.
\end{enumerate}
\end{definition}

\begin{lemma}\label{l3.6}{\rm{(\cite{dll07})}}
Let $0<\rho<n$. Suppose $\Omega(x,\,z)\in L^\fz(\rn) \times L^1(S^{n-1})$.
If there exists a constant $0<\bz\le1/2$ such that $|y|<\bz R$, then, for any $h\in\rn$,
\begin{align*}
\int_{R\le |x|<2R}\lf|\frac{\Omega(x+h,\,x-y)}{|x-y|^{n-\rho}}-\frac{\Omega(x+h,\,x)}{|x|^{n-\rho}}\r|\,dx
\leq CR^{\rho}\lf({\frac{|y|}{R}}+\int_{{2|y|}/{R}}^{{4|y|}/{R}}\frac{\omega(\delta)}{\delta}d\delta\r),
\end{align*}
where the positive constant $C$ is independent of $R$ and $y$.
\end{lemma}

\begin{proof}[Proof of Theorem \ref{dl.1}]
By the atomic decomposition theory of Hardy space (see \cite[Chapter 2]{l95}),
our problem reduces to prove that there exists a positive constant $C$ such that,
for any $(p,\,\fz,\,s)$-atom $a(x)$, $\|\mu_\Omega^\rho(a)\|_{L^p(\rn)}\le C$.
To this end, without loss of generality,
we may assume $a$ is supported in a ball $B:=B(\mathbf{0},\,r)$ for some $r\in(0,\,\fz)$.
Below, we estimate $\mu_\Omega^\rho(a)$ separately around and away from the support of atom $a(x)$ as follows.
To be precise, let us write
\begin{align*}
\int_{\rn}\lf|\mu_\Omega^\rho(a)(x)\r|^p\,dx=
\int_{8B}\lf|\mu_\Omega^\rho(a)(x)\r|^p\,dx+\int_{(8B)^\complement}\lf|\mu_\Omega^\rho(a)(x)\r|^p\,dx=:\mathrm{I+J}.
\end{align*}

For $\mathrm{I}$, by H\"{o}lder's inequality and Theorem A, we have
\begin{align*}
\mathrm{I}=\int_{8B}\lf|\mu_\Omega^\rho(a)(x)\r|^p\,dx
\le \lf(\int_{8B}\lf|\mu_\Omega^\rho(a)(x)\r|^2\,dx\r)^{p/2}|8B|^{1-p/2}
\ls \|a\|_{L^\fz(\rn)}^p|B|\ls 1.
\end{align*}

For $\mathrm{J}$, we rewrite
\begin{align*}
\mathrm{J}
&=\int_{(8B)^\complement}\lf|\mu_\Omega^\rho(a)(x)\r|^p\,dx \\
&=\int_{(8B)^\complement}\lf(\int_0^\fz\lf|\int_{|x-y|\le t}
 \frac{\Omega(x,\,x-y)}{|x-y|^{n-\rho}}a(y)\,dy\r|^2\frac{dt}{t^{2\rho+1}}\r)^{p/2}\,dx \\
&\le \int_{(8B)^\complement}\lf(\int_0^{|x|+2r}\lf|\int_{|x-y|\le t}
 \frac{\Omega(x,\,x-y)}{|x-y|^{n-\rho}}a(y)\,dy\r|^2\frac{dt}{t^{2\rho+1}}\r)^{p/2}\,dx \\
&\hs+\int_{(8B)^\complement}\lf(\int_{|x|+2r}^\fz\cdot\cdot\cdot\r)^{p/2}\,dx=:\mathrm{J_1+J_2}.
\end{align*}

We first estimate $\mathrm{J_1}$.
Noticing that $y\in B$ and $x\in {(8B)^\complement}$, we know that $|x-y|\sim|x|\sim|x|+2r$.
From this and the mean value theorem, it follows that, for any $y\in B$ and $x\in {(8B)^\complement}$,
\begin{align*}
\lf|\frac{1}{|x-y|^{2\rho}}-\frac{1}{(|x|+2r)^{2\rho}}\r|\ls\frac{r}{|x-y|^{2\rho+1}}.
\end{align*}
Using Minkowski's inequality for integrals and the above inequality, we know that
\begin{align*}
\mathrm{J_1}
&=\int_{(8B)^\complement}\lf(\int_0^{|x|+2r}\lf|\int_{|x-y|\le t}
 \frac{\Omega(x,\,x-y)}{|x-y|^{n-\rho}}a(y)\,dy\r|^2\frac{dt}{t^{2\rho+1}}\r)^{p/2}\,dx \\
&\leq \int_{(8B)^\complement}\lf[\int_{B}\lf|\frac{\Omega(x,\,x-y)}{|x-y|^{n-\rho}}a(y)\r|
 \lf(\int_{|x-y|}^{|x|+2r}\frac{dt}{t^{2\rho+1}}\r)^{1/2}\,dy\r]^p\,dx \\
&\ls |B|^{-1} \int_{(8B)^\complement}\lf[\int_{B}\frac{\lf|\Omega(x,\,x-y)\r|}{|x-y|^{n-\rho}}
 \lf|\frac{1}{|x-y|^{2\rho}}-\frac{1}{(|x|+2r)^{2\rho}}\r|^{1/2}\,dy\r]^p\,dx \\
&\ls r^{-n+p/2}\int_{(8B)^\complement}\lf(\int_{B}\frac{\lf|\Omega(x,\,x-y)\r|}{|x-y|^{n+1/2}}\,dy\r)^p\,dx.
\end{align*}
Thanks to $p>n/(n+1/2)$, we may choose $\varepsilon$ satisfying $0<\varepsilon<n+1/2-n/p$.
Apply H\"{o}lder's inequality to obtain
\begin{align*}
\mathrm{J_1}
&\ls r^{-n+p/2}\int_{(8B)^\complement}
 \lf(\int_{B}\frac{\lf|\Omega(x,\,x-y)\r|}{|x-y|^{n+\varepsilon}}\,dy\r)^p|x|^{(\varepsilon-1/2)p}\,dx \\
&\ls r^{-n+p/2}\lf(\int_{(8B)^\complement}
 \int_{B}\frac{\lf|\Omega(x,\,x-y)\r|}{|x-y|^{n+\varepsilon}}\,dy\,dx\r)^p
 \lf(\int_{(8B)^\complement}|x|^{(\varepsilon-1/2)\frac{p}{1-p}}\,dx\r)^{1-p} \\
&\sim r^{-n+p/2}\lf(\int_{|y|<r}
 \int_{|x|\ge8r}\frac{\lf|\Omega(x,\,x-y)\r|}{|x-y|^{n+\varepsilon}}\,dx\,dy\r)^p
 \lf(\int_{|x|\ge8r}|x|^{(\varepsilon-1/2)\frac{p}{1-p}}\,dx\r)^{1-p} \\
&\ls r^{-n+p/2}\lf(\int_{|y|<r}
 \int_{|x-y|>r}\frac{\lf|\Omega(x,\,x-y)\r|}{|x-y|^{n+\varepsilon}}\,dx\,dy\r)^p
 \lf(\int_{|x|>r}|x|^{(\varepsilon-1/2)\frac{p}{1-p}}\,dx\r)^{1-p} \\
&\sim r^{-n+p/2}\lf(\int_{|y|<r}
 \int_{|z|>r}\frac{\lf|\Omega(y+z,\,z)\r|}{|z|^{n+\varepsilon}}\,dz\,dy\r)^p
 \lf(\int_{|x|>r}|x|^{(\varepsilon-1/2)\frac{p}{1-p}}\,dx\r)^{1-p}   \\
&\sim r^{-n+p/2}\lf(\int_{|y|<r}
 \int_r^\fz u^{-\varepsilon-1}\,du\,dy\r)^p
 \lf(\int_r^\fz u^{(\varepsilon-1/2)\frac{p}{1-p}}u^{n-1}\,du\r)^{1-p}\sim1.
\end{align*}

Now we are interested in $\mathrm{J_2}$.
For any integer $j\ge 3$, denote simply $\{x\in\rn: \ 2^jr\le|x|<2^{j+1}r \}$ by $E_j$.
It is apparent from $t>|x|+2r$ that $B\subset \{y\in\rn: \ |x-y|\le t\}$.
From this, vanishing moments of atom $a(x)$, and Minkowski's inequality for integrals,
we deduced that
\begin{align*}
\mathrm{J_2}
&=\int_{(8B)^\complement}\lf[\int_{|x|+2r}^\fz\lf|\int_{|x-y|\le t}
 \lf(\frac{\Omega(x,\,x-y)}{|x-y|^{n-\rho}}-\frac{\Omega(x,\,x)}{|x|^{n-\rho}}\r)a(y)\,dy\r|^2
 \frac{dt}{t^{2\rho+1}}\r]^{p/2}\,dx \\
&\le \int_{(8B)^\complement}\lf[\int_{B}
 \lf|\frac{\Omega(x,\,x-y)}{|x-y|^{n-\rho}}-\frac{\Omega(x,\,x)}{|x|^{n-\rho}}\r||a(y)|
 \lf(\int_{|x|}^\fz\frac{dt}{t^{2\rho+1}}\r)^{1/2}\,dy\r]^p\,dx \\
&\sim \int_{(8B)^\complement}\lf(\int_{B}
 \lf|\frac{\Omega(x,\,x-y)}{|x-y|^{n-\rho}}-\frac{\Omega(x,\,x)}{|x|^{n-\rho}}\r|
 \frac{|a(y)|}{|x|^\rho}\,dy\r)^p\,dx \\
&\ls r^{-n}\sum_{j=3}^\fz\int_{E_j}\lf(\int_{B}
 \lf|\frac{\Omega(x,\,x-y)}{|x-y|^{n-\rho}}-\frac{\Omega(x,\,x)}{|x|^{n-\rho}}\r|
 \frac{1}{|x|^\rho}\,dy\r)^p\,dx \\
&\ls r^{-n}\sum_{j=3}^\fz(2^jr)^{n(1-p)}\lf(\int_{E_j}\int_{B}
 \lf|\frac{\Omega(x,\,x-y)}{|x-y|^{n-\rho}}-\frac{\Omega(x,\,x)}{|x|^{n-\rho}}\r|
 \frac{1}{|x|^\rho}\,dy\,dx\r)^p \\
&\ls r^{-np}\sum_{j=3}^\fz 2^{jn(1-p)} (2^jr)^{-\rho p}\lf(\int_{B}\int_{E_j}
 \lf|\frac{\Omega(x,\,x-y)}{|x-y|^{n-\rho}}-\frac{\Omega(x,\,x)}{|x|^{n-\rho}}\r|\,dx\,dy\r)^p.
\end{align*}
Using Lemma \ref{l3.6} and the assumption that $\Omega$ satisfies the $L^{1,\,\az}$-Dini condition,
the above inner integral is bounded by a positive constant times
\begin{align*}
(2^jr)^{\rho}\lf(\frac{|y|}{2^jr}+\int_{\frac{2|y|}{2^jr}}^{\frac{4|y|}{2^jr}}\frac{\omega(\delta)}{\delta}\,d\delta\r)
&\ls (2^jr)^{\rho}\lf[\frac{|y|}{2^jr}+
\lf(\frac{|y|}{2^jr}\r)^\az\int_{\frac{2|y|}{2^jr}}^{\frac{4|y|}{2^jr}}\frac{\omega(\delta)}{\delta^{1+\az}}\,d\delta\r] \\
&\ls  (2^jr)^{\rho}\lf[\frac{|y|}{2^jr}
 +\lf(\frac{|y|}{2^jr}\r)^\az\int_{0}^{1}\frac{\omega(\delta)}{\delta^{1+\az}}\,d\delta\r]
\ls (2^jr)^{\rho}2^{-j\az}.
\end{align*}
If we plug the above inequality into $\mathrm{J_2}$, we obtain that
\begin{align*}
\mathrm{J_2}\ls r^{-np}\sum_{j=3}^\fz 2^{jn(1-p)} (2^jr)^{-\rho p}\lf(\int_{|y|<r}(2^jr)^{\rho}2^{-j\az}\,dy\r)^p
\sim\sum_{j=3}^\fz 2^{j(n-np-\az p)}\sim1,
\end{align*}
where the last ``$\sim$" is due to $p>n/(n+\az)$.

Finally, collecting the estimates of ${\rm{I}}$, ${\rm{J_1}}$ and ${\rm{J_2}}$, we obtain the desired inequality.
This finishes the proof of Theorem \ref{dl.1}.
\end{proof}

\begin{proof}[Proof of Theorem \ref{dl.2}]
Since the proof of Theorem \ref{dl.2} is similar to that of Theorem \ref{dl.1},
we use the same notation as in the proof of Theorem \ref{dl.1}.
Rather that give a completed proof, we just give out the necessary modifications with respect to the estimate of ${\rm{J_2}}$. Rewrite
\begin{align*}
\mathrm{J_2}
&\le C\int_{(8B)^\complement}\int_{B}
 \lf|\frac{\Omega(x,\,x-y)}{|x-y|^{n-\rho}}-\frac{\Omega(x,\,x)}{|x|^{n-\rho}}\r|
 \frac{|a(y)|}{|x|^\rho}\,dy\,dx \\
&\le C\int_{(8B)^\complement}\int_{B}
 \lf|\frac{\Omega(x,\,x-y)}{|x|^\rho|x-y|^{n-\rho}}-\frac{\Omega(x,\,x-y)}{|x-y|^{n}}\r||a(y)|\,dy\,dx \\
&\hs+C\int_{(8B)^\complement}\int_{B}
 \lf|\frac{\Omega(x,\,x-y)}{|x-y|^{n}}-\frac{\Omega(x,\,x)}{|x|^{n}}\r||a(y)|\,dy\,dx=:C(\mathrm{J_{21}+J_{22}}).
\end{align*}

Below, we will give the estimates of $\mathrm{J_{21}}$ and $\mathrm{J_{22}}$, respectively.

For $\mathrm{J_{21}}$, noticing that $y\in B$ and $x\in {(8B)^\complement}$, we have
$|x|\sim|x-y|$ and $|y|\le|x-y|$.
From this and the mean value theorem, it follows that, for any $y\in B$ and $x\in {(8B)^\complement}$,
\begin{align*}
\lf|\frac{1}{|x|^{\rho}}-\frac{1}{|x-y|^{\rho}}\r|\ls\frac{r^{1/2}}{|x-y|^{\rho+{1/2}}}.
\end{align*}
Substituting the above inequality into ${\mathrm{J_{21}}}$, we have
\begin{align*}
\mathrm{J_{21}}
&\ls\int_{(8B)^\complement}\int_{B}
\frac{|\Omega(x,\,x-y)|}{|x-y|^{n-\rho}}\frac{r^{1/2}}{|x-y|^{\rho+{1/2}}}|a(y)|\,dy\,dx \\
&\ls r^{-n+{1/2}}\int_{|x|\ge8r}\int_{|y|<r}\frac{|\Omega(x,\,x-y)|}{|x-y|^{n+{1/2}}}\,dy\,dx \\
&\ls r^{-n+{1/2}}\int_{|y|<r}\int_{|x-y|>r}\frac{|\Omega(x,\,x-y)|}{|x-y|^{n+{1/2}}}\,dx\,dy \\
&\sim r^{-n+{1/2}}\int_{|y|<r}\lf(\int_{S^{n-1}}\int_r^\fz\frac{|\Omega(y+uz',\,z')|}{u^{n+{1/2}}}u^{n-1}\,du\,d\sigma(z')\r)dy \\
&\ls r^{-n+{1/2}}\int_{|y|<r}\lf(\int_r^\fz \frac{1}{u^{3/2}}\,du\r)dy\sim1.
\end{align*}

We are now turning to the estimate of $\mathrm{J_{22}}$. The H\"{o}rmander condition \eqref{p1} yields
\begin{align*}
\mathrm{J_{22}}
&=\int_{(8B)^\complement}\int_{B} \lf|\frac{\Omega(x,\,x-y)}{|x-y|^{n}}-\frac{\Omega(x,\,x)}{|x|^{n}}\r||a(y)|\,dy\,dx \\
&\le \int_{|y|<r} \lf(\int_{|x|\ge2|y|} \lf|\frac{\Omega(x,\,x-y)}{|x-y|^{n}}-\frac{\Omega(x,\,x)}{|x|^{n}}\r|dx\r)|a(y)|\,dy \\
&\ls \int_{|y|<r}|a(y)|\,dy\ls1.
\end{align*}

The proof is completed.
\end{proof}

\begin{proof}[Proof of Theorem \ref{dl.3}]
Proceeding as in the proof of \cite[Theorem 1.3]{dls04},
it is quite believable that this theorem may also be
true for parametric case, but to limit the length of this article,
we leave the details to the interested reader.
\end{proof}

\section{$WH^p$-$WL^p$ boundedness \label{s3}}
The main results of this section are as follows.
\begin{theorem}\label{dl.4}
Let $0<\rho<n$, $q>2(n-1)/n$, $0<\az\le1$, $\bz:=\min\{\az,\,1/2\}$ and $n/(n+\bz)<p<1$.
Suppose that $\Omega\in L^\fz \times L^q(S^{n-1})$ satisfies the $L^{1,\,\az}$-Dini condition.
Then $\mu_\Omega^\rho$ is bounded from $WH^p(\rn)$ to $WL^p(\rn)$.
\end{theorem}

\begin{theorem}\label{dl.5}
Let $0<\rho<n$ and $q>2(n-1)/n$. Suppose that $\Omega\in L^\fz \times L^q(S^{n-1})$.
If there exist two positive constants ${C}$ and $M$ such that,
for any $y$, $h\in\rn$,
\begin{align}\label{p2}
\int_{|x|\geq M|y|}\lf|\frac{\Omega(x+h,\,x-y)}{|x-y|^n}-\frac{\Omega(x+h,\,x)}{|x|^n}\r|\,dx\leq \frac{C}{M},
\end{align}
then $\mu_\Omega^\rho$ is bounded from $WH^1(\rn)$ to $WL^1(\rn)$.
\end{theorem}

\begin{theorem}\label{dl.6}
Let $0<\rho<n$ and $q>2(n-1)/n$.
Suppose $\Omega\in L^\fz \times L^q(S^{n-1})$. If
\begin{align}\label{p3}
\int_0^1\frac{\omega(\delta)}{\delta}(1+|\log{\delta}|)^\sz\,d\delta<\fz \ \mathrm{for \ some} \ \sz>1,
\end{align}
then $\mu_\Omega^\rho$ is bounded from $WH^1(\rn)$ to $WL^1(\rn)$.
\end{theorem}

\begin{corollary}\label{tl.7}
Let $0<\rho<n$ and $q>2(n-1)/n$.
Suppose $\Omega\in L^\fz \times L^q(S^{n-1})$. If \eqref{p2} or \eqref{p3} holds,
then $\mu_\Omega^\rho$ is bounded from $L^1(\rn)$ to $WL^1(\rn)$.
\end{corollary}

We need the following atomic decomposition theory of weak Hardy space.

\begin{lemma}\label{whp}{\rm{(\cite{l95})}}
Let $0<p\le1$. For every $f\in WH^p(\rn)$,
there exists a sequence of bounded measurable functions $\{f_k\}_{k=-\fz}^\fz$ such that
\begin{enumerate}
\item [\rm{(i)}] $f=\sum_{k=-\fz}^\fz f_k$ in the sense of distributions.

\item [\rm{(ii)}] Each $f_k$ can be further decomposed into $f_k=\sum_i b^k_i$ and $\{b^k_i\}$ satisfies

 {\rm{\quad(a)}} $\supp{(b^k_i)}\subset B^k_i:=B(x^k_i,\,r^k_i)$;
Moreover, $\sum_i \chi_{B^k_i}(x)\le C$ and $\sum_i |B^k_i|\le c\,2^{-kp}$, where $c\sim\|f\|^p_{WH^p(\rn)}$;

 {\rm{\quad(b)}} $\|b^k_i\|_{L^\fz(\rn)}\le C2^k$, where $C$ is independent of $k$ and $i$;

 {\rm{\quad(c)}} $\int_\rn b^k_i(x)x^\gz\,dx=0$ for any multi-index $\gz$ with $|\gz|\leq \lfloor n(1/p-1)\rfloor$.
\end{enumerate}

Conversely, if distribution $f$ has a decomposition satisfying $\mathrm{(i)}$ and $\mathrm{(ii)}$, then $f\in WH^p(\rn)$.
Moreover, we have $\|f\|^p_{WH^p(\rn)}\sim c$.
\end{lemma}

\begin{proof}[Proof of Theorem \ref{dl.4}]
To show Theorem \ref{dl.4}, it suffices to prove that
there exist a positive constant $C$ such that, for any $f\in WH^p(\rn)$ and $\lz\in(0,\,\fz)$,
\begin{align*}
\lf|\lf\{x\in\rn: \ \mu_\Omega^\rho(f)(x)>\lz\r\}\r|\le C{\lz^{-p}}{\|f\|_{WH^p(\rn)}^p}.
\end{align*}
To this end, we choose integer $k_0$ satisfying $2^{k_0}\le\lz<2^{k_0+1}$.
By Lemma \ref{whp}, we may write
\begin{align*}
f=\sum_{k=-\fz}^{k_0}\sum_{i}b^k_i+\sum_{k=k_0+1}^{\fz}\sum_i b^k_i=:F_1+F_2,
\end{align*}
where $b^k_i$ satisfies (a), (b) and (c) of Lemma \ref{whp}.

We estimate $F_1$ first. For $F_1$, we claim that $\|F_1\|_{L^2(\rn)}\ls\lz^{1-{p}/{2}}\|f\|^{{p}/{2}}_{WH^p(\rn)}$.
In fact, by Minkowski's inequality and the finite overlapped property of $\{B^k_i\}$, we obtain that
\begin{align*}
\|F_1\|_{L^2(\rn)}
&\le\sum_{k=-\fz}^{k_0}\sum_{i}\lf\|b^k_i\r\|_{L^2(\rn)}
\le\sum_{k=-\fz}^{k_0}\sum_{i}\lf\|b^k_i\r\|_{L^\fz(\rn)}\lf|B^k_i\r|^{1/2} \\
&\ls\sum_{k=-\fz}^{k_0}2^k\lf(\sum_{i}\lf|B^k_i\r|\r)^{1/2}\ls\sum_{k=-\fz}^{k_0}2^{k(1-{p}/{2})}\|f\|_{WH^p(\rn)}^{{p}/{2}}
\sim\lz^{(1-{p}/{2})}\|f\|_{WH^p(\rn)}^{{p}/{2}}.
\end{align*}
From Theorem A and the above claim, we deduce that
\begin{align*}
\lf|\lf\{x\in\rn: \ \mu_\Omega^\rho(F_1)(x)>\lz\r\}\r|
&\le \lz^{-2}\lf\|\mu_\Omega^\rho(F_1)\r\|^2_{L^2(\rn)} \\
&\ls \lz^{-2}\lf\|F_1\r\|^2_{L^2(\rn)}\ls\lz^{-p}\|f\|_{WH^p(\rn)}^p.
\end{align*}

Next let us deal with $F_2$. Set
\begin{align*}
A_{k_0}:=\bigcup_{k=k_0+1}^\fz\bigcup_i\tilde{B_i^k}\, ,
\end{align*}
where
$\tilde{B_i^k}:=B(x^k_i,\,8(3/2)^{{(k-k_0)p}/{n}}\,r^k_i)$.
To show that
\begin{align*}
\lf|\lf\{x\in\rn: \ \mu_\Omega^\rho(F_2)(x)>\lz\r\}\r|\ls{\lz^{-p}}{\|f\|_{WH^p(\rn)}^p},
\end{align*}
we cut $|\{x\in\rn: \ \mu_\Omega^\rho(F_2)(x)>\lz\}|$ into $A_{k_0}$ and
$\{x\in (A_{k_0})^\complement: \ \mu_\Omega^\rho(F_2)(x)>\lz\}$.

For $A_{k_0}$, a routine computation gives rise to
\begin{align*}
|A_{k_0}|
&\le \sum_{k=k_0+1}^{\fz}\sum_i\lf|\tilde{B_i^k}\r|\sim\sum_{k=k_0+1}^{\fz}\sum_i\lf(\frac{3}{2}\r)^{(k-k_0)p}\lf|{B_i^k}\r| \\
&\ls\sum_{k=k_0+1}^{\fz}\lf(\frac{3}{2}\r)^{(k-k_0)p}2^{-kp}\|f\|^p_{WH^p(\rn)}\sim\lz^{-p}\|f\|_{WH^p(\rn)}^p.
\end{align*}

It remains to estimate $(A_{k_0})^\complement$.
Applying the inequality $\|\cdot\|_{\ell^1}\le\|\cdot\|_{\ell^p}$ with $p\in(0,\,1)$, we conclude that
\begin{align*}
\lz^p\lf|\lf\{x\in{\lf(A_{k_0}\r)^\complement}: \ \mu_\Omega^\rho(F_2)(x)>\lz\r\}\r|
&\le \int_{\lf(A_{k_0}\r)^\complement}\lf|\mu_\Omega^\rho(F_2)(x)\r|^p\,dx \\ \nonumber
&\le \int_{\lf(A_{k_0}\r)^\complement}\sum_{k=k_0+1}^{\fz}\sum_i\lf|\mu_\Omega^\rho(b^k_i)(x)\r|^p\,dx \\ \nonumber
&\le \sum_{k=k_0+1}^{\fz}\sum_i\int_{\lf({\tilde{B_i^k}}\r)^\complement}\lf|\mu_\Omega^\rho(b^k_i)(x)\r|^p\,dx \\ \nonumber
&=:\sum_{k=k_0+1}^{\fz}\sum_i(\mathrm{K_1+K_2}),
\end{align*}
where
\begin{align*}
\mathrm{K_1}=\int_{\lf({\tilde{B_i^k}}\r)^\complement}\lf(\int_0^{|x-x^k_i|+2r^k_i}\lf|\int_{|x-y|\le t}
 \frac{\Omega(x,\,x-y)}{|x-y|^{n-\rho}}b^k_i(y)\,dy\r|^2\frac{dt}{t^{2\rho+1}}\r)^{{p}/{2}}\,dx
\end{align*}
and
\begin{align*}
\mathrm{K_2}=\int_{\lf({\tilde{B_i^k}}\r)^\complement}\lf(\int^\fz_{|x-x^k_i|+2r^k_i}\lf|\int_{|x-y|\le t}
 \frac{\Omega(x,\,x-y)}{|x-y|^{n-\rho}}b^k_i(y)\,dy\r|^2\frac{dt}{t^{2\rho+1}}\r)^{{p}/{2}}\,dx.
\end{align*}

The estimates of $\mathrm{K_1}$ and $\mathrm{K_2}$
are quite similar to that given earlier for $\mathrm{J_1}$ and $\mathrm{J_2}$ in Theorem \ref{dl.1}, respectively,
and hence no proof will be given here.
We directly give the estimate of $\mathrm{K_1}+\mathrm{K_2}$ below,
\begin{align*}
\mathrm{K_1}+\mathrm{K_2}\ls2^{kp}\lf|B^k_i\r|\lf(\frac{2}{3}\r)^{\frac{p(pn+p\bz-n)}{n}(k-k_0)},
\end{align*}
which, together with $p>n/(n+\bz)$, implies that
\begin{align*}
\lz^p\lf|\lf\{x\in{\lf(A_{k_0}\r)^\complement}: \ \mu_\Omega^\rho(F_2)(x)>\lz\r\}\r|
&\ls\sum_{k=k_0+1}^{\fz}\sum_i(\mathrm{K_1+K_{2}}) \\
&\ls\sum_{k=k_0+1}^{\fz}\sum_i 2^{kp}\lf|B^k_i\r|\lf(\frac{2}{3}\r)^{\frac{p(np+\bz p-n)}{n}(k-k_0)} \\
&\ls\|f\|^p_{WH^p(\rn)}\sum_{k=k_0+1}^{\fz}\lf(\frac{2}{3}\r)^{\frac{p(np+\bz p-n)}{n}(k-k_0)}\\
&\sim\|f\|^p_{WH^p(\rn)}.
\end{align*}

The proof is completed.
\end{proof}

\begin{proof}[Proof of Theorem \ref{dl.5}]
Since the proof of Theorem \ref{dl.5} is similar to that of Theorem \ref{dl.4},
we use the same notation as in the proof of Theorem \ref{dl.4}.
Rather that give a completed proof,
we just give out the necessary modifications with respect to the estimates of ${\rm{K_1}}$ and ${\rm{K_2}}$.

For ${\rm{K_1}}$, it follows from Minkowski's inequality and the mean value theorem that
\begin{align*}
{\rm{K_1}}
&=\int_{\lf({\tilde{B_i^k}}\r)^\complement}\lf(\int_0^{|x-x^k_i|+2r^k_i}\lf|\int_{|x-y|\le t}
 \frac{\Omega(x,\,x-y)}{|x-y|^{n-\rho}}b^k_i(y)\,dy\r|^2\frac{dt}{t^{2\rho+1}}\r)^{1/2}\,dx \\
&\leq \int_{\lf({\tilde{B_i^k}}\r)^\complement}\lf[\int_{B^k_i}\lf|\frac{\Omega(x,\,x-y)}{|x-y|^{n-\rho}}b^k_i(y)\r|
 \lf(\int_{|x-y|}^{|x-x^k_i|+2r^k_i}\frac{dt}{t^{2\rho+1}}\r)^{{1/2}}\,dy\r]\,dx \\
&\ls 2^k \int_{\lf({\tilde{B_i^k}}\r)^\complement}\lf(\int_{B^k_i}\frac{\lf|\Omega(x,\,x-y)\r|}{|x-y|^{n-\rho}}
 \lf|\frac{1}{|x-y|^{2\rho}}-\frac{1}{({|x-x^k_i|+2r^k_i})^{2\rho}}\r|^{{1/2}}\,dy\r)\,dx \\
&\ls 2^k \lf(r^k_i\r)^{1/2}\int_{\lf({\tilde{B_i^k}}\r)^\complement}
 \lf(\int_{B^k_i}\frac{\lf|\Omega(x,\,x-y)\r|}{|x-y|^{n+{1/2}}}\,dy\r)\,dx,
\end{align*}
which, together with the same argument as that used in ${\rm{J_{21}}}$ of Theorem \ref{dl.2}, implies that
\begin{align*}
{\rm{K_1}} \ls 2^k \lf|B^k_i\r|\lf(\frac{2}{3}\r)^\frac{k-k_0}{2n}.
\end{align*}

For ${\rm{K_2}}$, it is apparent from $t>|x-x^k_i|+2r^k_i$ that $B^k_i\subset \{y\in\rn: \ |x-y|\le t\}$.
From this, vanishing moments of $b^k_i$, and Minkowski's inequality for integrals,
we deduced that
\begin{align*}
\mathrm{K_2}
&=\int_{\lf({\tilde{B_i^k}}\r)^\complement}\lf[\int_{|x-x^k_i|+2r}^\fz\lf|\int_{|x-y|\le t}
 \lf(\frac{\Omega(x,\,x-y)}{|x-y|^{n-\rho}}-\frac{\Omega(x,\,x-x^k_i)}{|x-x^k_i|^{n-\rho}}\r)b^k_i(y)\,dy\r|^2
 \frac{dt}{t^{2\rho+1}}\r]^{1/2}\,dx \\
&\le \int_{\lf({\tilde{B_i^k}}\r)^\complement}\lf[\int_{B^k_i}
 \lf|\frac{\Omega(x,\,x-y)}{|x-y|^{n-\rho}}-\frac{\Omega(x,\,x-x^k_i)}{|x-x^k_i|^{n-\rho}}\r|\lf|b^k_i(y)\r|
 \lf(\int_{|x-x^k_i|}^\fz\frac{dt}{t^{2\rho+1}}\r)^{1/2}\,dy\r]\,dx \\
&\le C\int_{\lf({\tilde{B_i^k}}\r)^\complement}\lf(\int_{B^k_i}
 \lf|\frac{\Omega(x,\,x-y)}{|x-y|^{n-\rho}}-\frac{\Omega(x,\,x-x^k_i)}{|x-x^k_i|^{n-\rho}}\r|
 \frac{\lf|b^k_i(y)\r|}{|x-x^k_i|^\rho}\,dy\r)\,dx \\
&\le C\int_{\lf({\tilde{B_i^k}}\r)^\complement}\int_{B^k_i}
 \lf|\frac{\Omega(x,\,x-y)}{|x-x^k_i|^\rho|x-y|^{n-\rho}}-\frac{\Omega(x,\,x-y)}{|x-y|^{n}}\r|\lf|b^k_i(y)\r|\,dy\,dx \\
&\hs+C\int_{\lf({\tilde{B_i^k}}\r)^\complement}\int_{B^k_i}
 \lf|\frac{\Omega(x,\,x-y)}{|x-y|^{n}}-\frac{\Omega(x,\,x-x^k_i)}{|x-x^k_i|^{n}}\r|\lf|b^k_i(y)\r|\,dy\,dx=:C(\mathrm{K_{21}+K_{22}}).
\end{align*}

Below, we will give the estimates of $\mathrm{K_{21}}$ and $\mathrm{K_{22}}$, respectively.

For ${\rm{K_{21}}}$, an argument similar to the one used in ${\rm{J_{21}}}$ of Theorem \ref{dl.2} shows that
\begin{align*}
{\rm{K_{21}}}\ls2^k \lf|B^k_i\r|\lf(\frac{2}{3}\r)^\frac{k-k_0}{2n}.
\end{align*}

We are now turning to the estimate of $\mathrm{J_{22}}$. The H\"{o}rmander-type condition \eqref{p2} yields
\begin{align*}
\mathrm{K_{22}}
&=\int_{\lf({\tilde{B_i^k}}\r)^\complement}\int_{B^k_i}
\lf|\frac{\Omega(x,\,x-y)}{|x-y|^{n}}-\frac{\Omega(x,\,x-x^k_i)}{|x-x^k_i|^{n}}\r|\lf|b^k_i(y)\r|\,dy\,dx \\
&\ls2^k\int_{|y-x^k_i|<r^k_i}\int_{|x-x^k_i|\ge8(3/2)^\frac{k-k_0}{n}r^k_i}
\lf|\frac{\Omega(x,\,x-y)}{|x-y|^{n}}-\frac{\Omega(x,\,x-x^k_i)}{|x-x^k_i|^{n}}\r|\,dx\,dy \\
&\sim 2^k\int_{|y|<r^k_i}\int_{|x|\ge8(3/2)^\frac{k-k_0}{n}r^k_i}
\lf|\frac{\Omega(x+x^k_i,\,x-y)}{|x-y|^{n}}-\frac{\Omega(x+x^k_i,\,x)}{|x|^{n}}\r|\,dx\,dy \\
&\ls 2^k\int_{|y|<r^k_i}\lf(\frac23\r)^\frac{k-k_0}{n}\,dy \ls2^k \lf|B^k_i\r|\lf(\frac{2}{3}\r)^\frac{k-k_0}{2n}.
\end{align*}

Collecting the estimates of $\mathrm{K_{1}}$, $\mathrm{K_{21}}$ and $\mathrm{K_{22}}$, we know that
\begin{align*}
\mathrm{K_{1}}+\mathrm{K_{21}}+\mathrm{K_{22}}\ls2^k \lf|B^k_i\r|\lf(\frac{2}{3}\r)^\frac{k-k_0}{2n},
\end{align*}
which implies that
\begin{align*}
\lz\lf|\lf\{x\in{\lf(A_{k_0}\r)^\complement}: \ \mu_\Omega^\rho(F_2)(x)>\lz\r\}\r|
&\ls\sum_{k=k_0+1}^{\fz}\sum_i(\mathrm{K_1+K_{21}+K_{22}}) \\
&\ls\sum_{k=k_0+1}^{\fz}\sum_i2^k \lf|B^k_i\r|\lf(\frac{2}{3}\r)^\frac{k-k_0}{2n} \\
&\sim\|f\|_{WH^1(\rn)}.
\end{align*}

The proof is completed.
\end{proof}

\begin{proof}[Proof of Theorem \ref{dl.6}]
Proceeding as in the proof of \cite[Theorem 1.5]{dls04},
it is quite believable that Theorem \ref{dl.6} may also be
true for parametric case, but to limit the length of this article,
we leave the details to the interested reader.
\end{proof}

\begin{proof}[Proof of Corollary \ref{tl.7}]
By an argument similar to that used in \cite[Remark 1.8]{dls04},
we can easily carry out the proof of this corollary, the details being omitted.
\end{proof}


%
%
%
%
%


\noindent Li Bo

\medskip

%

\noindent{E-mail }: \texttt{bli.math@outlook.com}

\bigskip

\end{document}